\documentclass[11pt]{amsart}
\usepackage{epstopdf}
\usepackage[caption=false]{subfig}

\usepackage[numbers,sort&compress]{natbib}
\bibpunct[, ]{[}{]}{,}{n}{,}{,}
\makeatletter
\def\NAT@def@citea{\def\@citea{\NAT@separator}}
\makeatother

\newtheorem{theorem}{Theorem}[section]

\newtheorem{definition}[theorem]{Definition}
\newtheorem{example}[theorem]{Example}

\newtheorem{remark}{Remark}

\usepackage{graphicx}

\usepackage{bigints}
\usepackage{todonotes}
\presetkeys{todonotes}{color=orange!30}{}
\usepackage{color}

\usepackage{mathrsfs} 
\usepackage{eucal}	
\usepackage{amssymb}

\usepackage{tikz} 
\usepackage{amsmath} 
\usepackage{cancel}
\usepackage{enumitem} 

\usepackage{hyperref} 
\usepackage{varioref} 

\usepackage{url}

\usepackage{color}

\newcommand{\numberset}{\mathbb} 
\newcommand{\R}{\numberset{R}}
\newcommand{\N}{\numberset{N}}
\newcommand{\Z}{\numberset{Z}}

\newcommand{\beautyset}{\mathcal} 

\newcommand{\M}{\beautyset{M}}
\newcommand{\T}{\beautyset{T}}

\newcommand{\F}{\beautyset{F}}

\newcommand{\J}{\mathbf{j}}
\newcommand{\V}{\beautyset{V}}

\newcommand{\del}{\mathfrak{d}}

\def\Ss{{\bf t}}

\def\i{\mathfrak i}
\def \ds{\displaystyle}
\def \vsm{\vskip 0.2 truecm}
\def \vsmm{\vskip 0.1 truecm}

\def\bel{\begin{equation}\label}
\def\eeq{\end{equation}}
\def \w{\omega}
\def \d{{\bf d}}

\def\sgn{\text{ {\rm sgn}}}

\begin{document}
 
\title{A Lie-bracket-based notion of stabilizing feedback   in optimal control}
\author{Giovanni Fusco}\address{G. Fusco, Dipartimento di Matematica,
Universit\`a di Padova\\ Via Trieste, 63, Padova  35121, Italy\\
email:\,
fusco@math.unipd.it}
\author{Monica Motta}\address{M. Motta, Dipartimento di Matematica,
Universit\`a di Padova\\ Via Trieste, 63, Padova  35121, Italy\\
email:\,
motta@math.unipd.it}
\author{Franco Rampazzo}\address{F. Rampazzo, Dipartimento di Matematica,
Universit\`a di Padova\\ Via Trieste, 63, Padova  35121, Italy\\
email:\,
rampazzo@math.unipd.it}

\maketitle
\begin{abstract}
 For a   control system   two major issues  can be considered: the stabilizability  with respect to a given target,  and the minimization   of an integral functional (while the trajectories reach this target).
	Here we consider a problem where  stabilizability or controllability are investigated together with the further aim of a `cost regulation', namely a state-dependent upper bounding  of the functional. This paper is devoted to a crucial step in the program of establishing a chain of equivalences among {\it degree-$k$ stabilizability with regulated cost},  {\it asymptotic controllability with regulated cost},  and  the existence of a {\it degree-$k$ Minimum Restraint Function} (which is a special kind of Control Lyapunov Function). Besides the presence of a cost  we allow the stabilizing `feedback'   to give rise to directions that range in the union of original directions  and the family of iterated Lie bracket of length $\leq k$.  
	In the main result  asymptotic controllability [resp. with regulated cost] is proved to be necessary for  degree-$k$ stabilizability [resp. with regulated cost].  Further steps of the above-mentioned logical chain are proved in companion papers, so that also a Lyapunov-type inverse theorem —i.e. the possibility of deriving existence of a Minimum Restraint Function from stabilizability— appears as quite likely. 
	 
\end{abstract}

\subsection{Keywords} \texttt{Asymptotic stabilizability, asymptotic controllability, discontinuous feedback laws, optimal control, Lie brackets.}

\subsection{AMS Subject classification codes}\texttt{93B05 \and 93B27  \and 93B52 \and 93D20}.

\section{Introduction}

In order to  describe the issues investigated in the present article, let us briefly recall the  classical concepts of  asymptotic controllability,   stabilizability, and  Control Lyapunov Function,  with respect to a  control system  
 \bel{sys_intro}
 \dot y  = f(y ,a),\qquad y\in\R^n, \  a\in A\subset\R^m.
  \eeq 
   The open-loop  notion of {\it asymptotic controllability} to  a target $\T \subset \R^n$ prescribes that,  for any initial condition  $x\in \R^n\backslash \T$, there exist an exit-time $0<S_y \leq+\infty$ and  a control $\alpha:[0,S_y[\to A$ such that the corresponding solution $y:[0,S_y[\to \Bbb R$ of  \eqref{sys_intro} with $y(0)=x$   verifies $y([0,S_y[)\subset \R^n\backslash \T$ and  $\lim_{s\to S_y^-}{\rm dist}(y(s),\T) = 0.$  Furthermore, the so-called {\it overshoot boundedness 
} and {\it uniform attractiveness} properties have to be satisfied: namely,  the trajectory must   remain in a bounded domain  and will reach a given neighbourhood of the target after a certain time, respectively,  in  a suitable uniform way with respect to the initial condition $x$.  

The notion of {\it stabilizability}, which  is the closed-loop counterpart of asymptotic controllability, is variously defined in literature, but essentially consists in the existence of a feedback control $\R^n\backslash \T \ni y\mapsto \alpha(y)\in A$ whose corresponding trajectories ---namely the solutions of the ordinary differential equation $ \dot y = f(y,\alpha(y))$---  approach the target, still in a suitable uniform way. Since
  $\alpha(\cdot)$   may happen to  be discontinuous \cite{Brockett,SS80,So,C10}, the notion of `solution' needs further specifications: in this paper we will adopt the kind of solution associated to {\it sample-stabilizability} \cite{CLSS,CLRS}, even though other choices might be perfectly reasonable \cite{AB}.    

Finally, the existence of a {\it Control Lyapunov Function} (CLF), namely a (suitably weak) solution $U:\R^n\backslash \T\to ]0,+\infty[$ to  the Hamilton-Jacobi type,  dissipative   partial differential inequality 
\bel{HJI}   H(y,DU)<0,\qquad\quad\,\,\,  H(y,p) := \ds \min_{a\in A}\,\, \langle p, f(y,a)\rangle, \eeq
allows one to build a stabilizing feedback by choosing a selection
$$
y\mapsto \alpha(y)\in \ds \underset{a\in A}{\text{argmin}}  \,\langle DU, f(y,a) \rangle.
$$
This can be regarded  as  the first step  of the following  important circular chain of implications: 
\bel{circle}\framebox{\text{\bf existence of a} CLF} \implies \framebox{\text{\bf  stab.}}\implies
\framebox{\text{\bf  as. contr.}} \implies \framebox{\text{\bf existence of a {\rm CLF}}},  \eeq  where $\text{\bf stab.}$ and $\text {\bf as. contr.}$ stand for `stabilizability' and `asymptotic controllability', respectively.
A vast literature has been devoted to  prove each of the above implications,   and to date the whole logical chain  in \eqref{circle} has been established  under various sets of assumptions. 
Let us focus on the implication  \bel{stacon}\framebox{\text{\bf  stab.}}\implies
\framebox{\text{\bf  as. contr.}},\eeq
whose theoretical importance lies on the  fact that it plays crucial in the so-called `inverse Lyapunov problem'. The latter consists in the aim to prove that  the existence of a CLF is also necessary  for   stabilizability. However, it is somehow more natural to prove that asymptotic controllability, rather than stabilizability, is sufficient for the existence of a CLF, so that   establishing \eqref{stacon}   proves crucial.  Actually, in the standard framework   \eqref{stacon} can be easily proved.

However, in  the  high order optimization-stabilization   problem described below,     implication \eqref{stacon} is nothing but trivial, and,  as a matter of fact, it   constitutes  the main result of this paper. 
 To illustrate it,   let us begin by briefly describing what we mean by the  {\it high order optimization-stabilization issue.}
\begin{enumerate}
	
	\item {\it The  stabilization-optimization issue.} On the one hand, concomitantly to the reachability of the target $\T$ one might  aim at minimizing    an integral cost
	$$
	  \,\,\,\,\int_0^{S_y} l(y(s),\alpha(s)) \, ds, \quad {(l\ge0)}.$$  
Accordingly, one defines  a particular CLF, called {\it Minimum Restraint Function} (MRF) as a (suitably weak) solution of the  Hamilton-Jacobi  type, dissipative  differential inequality
$$ \hat H(y,p_0, DU)<0, \qquad\quad \,\,\, \hat H(y,p_0,p) :=\ds \min_{a\in A} \,\langle p, f(y,a) + p_0\, l(y,a)\rangle,$$
	  the `cost multiplier' $p_0>0$ being suitably chosen \cite{MR,LMR}.
	By constructing a feedback law as  a selection 
	$$
	\ds  y\mapsto \hat \alpha(y)\in 
	\underset{a\in A}{\text{argmin}}  \,\langle p\cdot f(y,a) + p_0\, l(y,a)\rangle, 
	$$
	one obtains both  stabilizability and a cost bound,  a property which is  called {\it stabilizability with regulated cost} \cite{LM,LM2,LM19}.
	
	\item {\it  The higher order stabilization issue.}  On the other hand, considering  nonlinear systems with   control-linear dynamics\footnote{The possibility of considering more general systems is discussed in \cite[Sec. 6]{FMR2}.}  
	$$
	f(y,a): =f_1(y)a^1+\dots+f_m(y)a^m,
	$$
 for every $k\geq 2$  we  manage to  include   the iterated Lie brackets of the vector fields $f_1,\dots,f_m$ up to the length $k$ in the  Hamiltonians   defining the   dissipative partial differential inequalities  (both in the purely dynamical case and in the presence of  an integral cost). The obvious  idea behind this generalization can be expressed by saying that   one regards iterated Lie brackets as   hidden (higher order) dynamics, which can be suitably approximated (in an appropriate  time scale).\end{enumerate}

Aiming to a generalization which includes both issues (1) and (2), we introduce, for every natural number $k$,   degree-$k$ Hamiltonians  $H^{(k)}(y,p_0,p)$  containing the current cost  $l$ and   obtained as  arg-min-max over  Lie-brackets of length $h\leq k$ and  control values $a\in A$ (see \cite{FMR2} and  \cite{MR2,Fu,MR3}).
Under suitable  assumptions, the solutions $U: \Bbb R^n\backslash \T \to ]0,+\infty[$ to the Hamilton-Jacobi type,  dissipative  differential inequality 
\bel{HJIk} H^{(k)}(y,p_0(U),DU)<0, \eeq
(where $p_0:\Bbb R_{\geq 0}\to [0,1]$ is now  a continuous map) are called {\it degree-$k$ $p_0$-MRF functions}. In \cite{FMR2} we have investigated the possibility, once a function verifies \eqref{HJIk}, of constructing a stabilizing feedback  by an  arg-min-max argument similar to the above one. Let us point out that   at some points $y$ such a feedback  may  select an iterated  Lie bracket $B(y)$ ---whose   length $\ell_{B(y)}$ verifies  $1<\ell_{B(y)}\leq k$--- rather than the vector fields $\pm f_1,\dots,\pm f_m$. Since the  order of  the  displacements associated to a Lie bracket $B(y)$ increases with the length $\ell_{B(y)}$,  the resulting notion of stabilizability (with  regulated cost) turns out to be a bit more demanding than the usual one,\footnote{Instead, for mere stabilization without a cost, the new notion of stabilizability is equivalent to the classical one, see Theorem \ref{L_equiv}.} in that it includes  sample intervals whose amplitude depends (at each $y$) on the lenght $\ell_{B(y)}$  of the  bracket $B(y)$ selected by the associated feedback.  
As already mentioned,  the possibility of getting  the existence of a    $p_0$-MRF function by the assumption of global asymptotic controllability seems quite likely (for the case without cost, see \cite{KT04,LM3,CLSS,R1}). Hence,  it becomes  crucial to prove  implication \eqref{stacon} for our extended setting: this is actually the content of our main result  (Theorem \ref{sample-->gaccosto}).
The paper is concluded by  a result of equivalence between the considered notion of stabilizability (without a cost) and  both the standard one and  the one considered in \cite{Fu}. In particular this tells us that the extra complication of our  definition of stabilizability is only due to the aim of {\it regulating}, i.e. bounding, the cost.

\vsm
In the remaining part of this Section we introduce some general  notations and definitions. In Section \ref{section_mrf} we give precise definitions of { \it degree-$k$ feedback generator}, {\it  sampling process},   and {\it degree-$k$ sample stabilizability with regulated cost}. The main result is established in Section \ref{concludesec}. Finally, for the purely dynamical case (i.e. with no cost),  in Section \ref{sec_comparison} we establish the equivalence  among  three notions of stabilizability: the one  proposed here, the standard one \cite{CLSS}, and the one recently introduced in \cite{Fu}.

\subsection{Notation and preliminaries} \label{preliminari}
For any  $a,b\in\R$, let us set  $a\vee b:= \max\{a,b\}$, $a\wedge b:= \min\{a,b\}$. For any integer $N\ge1$ 
we set $\R_{\ge0}^N:=[0,+\infty[^N$ and $\R_{>0}^N:=]0,+\infty[^N$. For  $N=1$ we simply write $\R_{\ge0}$ and  $\R_{>0}$, respectively. 
Given a nonempty set $X\subset\R^n$, we write $\overline X$ for the closure of $X$. 
Given an  open, nonempty subset $\Omega \subseteq \R^n$ and an integer $k\geq 1$, we write $C^k(\Omega)$ for the set of vector fields of class $C^k$ on $\Omega$, namely $C^k(\Omega):=C^k(\Omega; \R^n)$, while  $C^{k}_b(\Omega) \subset C^k(\Omega)$ denotes the subset
 of vector fields with bounded derivatives up to order $k$.
We use  $C^{k-1,1}(\Omega)\subset C^{k-1}(\Omega)$ to denote the subset of vector fields whose $(k-1)$-th derivative is Lipschitz continuous on $\Omega$, and we set  and $C^{k-1,1}_b (\Omega):=C^{k-1}_b(\Omega)\cap C^{k-1,1}(\Omega) $. 


\noindent  We say that a continuous function $G:\overline{\Omega} \to\R$ is  {\em positive definite} if  $G(x)>0$ \,$\forall x\in\Omega$ and $G(x)=0$ for any $x$ belonging to the boundary of $\Omega$. The function $G$ is called {\em proper}  if the pre-image $G^{-1}({\mathcal K})$ of any compact set ${\mathcal K}\subset\R_{\geq0}$ is compact.   
For a continuous function $G: \overline{\R^n\setminus\T} \to\R$ which is positive definite and proper, let us define 
the function $d_{G_-}:\R_{\ge0} \to \R_{\ge0}$ as follows: 
\bel{zeta}	
\begin{aligned}
&d_{G_-}(u) := \inf \big\{ {\bf d}(x) \text{ : } x\in \overline{\R^n\setminus\T} \ \ {\rm with} \ \ G(x) \geq u \big\}. 
\end{aligned}
\eeq
It is easy to see that $d_{G_-}$ is strictly increasing,   
$\lim_{u\to0^+}d_{G_-}(u)=0=d_{G_-}(0)$,
 and
\bel{Lzeta}
d_{G_-}(G(x))\le \d(x) \qquad \forall x\in \overline{\R^n\setminus\T}.
\eeq
Moreover, since \eqref{Lzeta} is the property  of $d_{G_-}$  that we are interested in,  we can assume without loss of generality $d_{G_-}$ continuous, by possibly replacing it with a continuous, strictly increasing function with value zero at zero, which approximates $d_{G_-}$ from below and still satisfies \eqref{Lzeta}.

We say that a function $G: \Omega \to \R$ is \textit{semiconcave} (with linear modulus) on $\Omega$ if it is continuous and for any closed subset $\M\subset\Omega$ there exists $\eta_{_\M} >0$ such that  
\[  G(x_1) + G(x_2) - 2G\left( \frac{x_1+x_2}{2}\right) \leq \eta_{_\M} |x_1-x_2|^2      \]
for all $x_1$, $x_2 \in \M$ such that the segment $\{\lambda x_1+(1-\lambda)x_2  \text{ : } \lambda\in[0,1]\}$ is contained in $\M$. If this property is valid just for any compact subset $\M\subset\Omega$, $G$ is said to be \textit{locally semiconcave} (with linear modulus) on $\Omega$.

\vsm 
Finally, let us collect some basic definitions on iterated Lie brackets.  If $g_1$, $g_2$ are $C^1$ vector fields on $\R^N$  the {\it Lie bracket of $g_1$ and $g_2$} is defined  as
$$
[g_1,g_2](x) := Dg_2(x)\cdot g_1(x) -  D g_1(x)\cdot g_2(x)\,\,\, \big(= - [g_2,g_1](x)\big).
$$ 
As is well-known, the map $[g_1,g_2]$ is a true
 vector field, i.e. it can be defined intrinsically. If the vector fields are sufficiently regular,  one can iterate the bracketing process: for instance, given a $4$-tuple ${\bf g}:=(g_1,g_2,g_3,g_4)$ of vector fields  one can construct  the brackets $[[g_1,g_2],g_3]$,  $[[g_1,g_2],[g_3,g_4]]$,  $[[[g_1,g_2],g_3],g_4]$, $[[g_2,g_3],g_4]$.   Accordingly, one can consider the  {\it  (iterated)  formal brackets } $B_1:=[[X_1,X_2],X_3]$, $B_2:=[[X_1,X_2],[X_3,X_4]]$, $B_3:=[[[X_1,X_2],X_3],X_4]$, $B_4:=[[X_2,X_3],X_4
 ]$ (regarded as  sequence of letters $X_1,\ldots,X_4$, commas, and left and right square  parentheses), so that, with obvious meaning of the notation,   $B_1({\bf g}) = [[g_1,g_2],g_3]$,   $B_2({\bf g}) = [[g_1,g_2],[g_3,g_4]]$, $B_3({\bf g}) =[[[g_1,g_2],g_3],g_4]$, $B_4({\bf g}) =[[g_2,g_3],g_4]$.

 \noindent The {\it degree} (or {\it length}) of a  formal bracket is   the number $\ell_{_B}$ of letters that are   involved in it.  For instance, the brackets $B_1, B_2, B_3, B_4$ have degrees equal to  $3$, $4$, $4$, and $3$, respectively.  By convention,  a single  variable $X_i$  is a formal bracket of degree $1$. {Given a formal bracket $B$ of degree $\ge2$, then there exist formal brackets $B_1$ and $B_2$ such that $B=[B_1,B_2]$. The pair $(B_1,B_2)$ is univocally determined and it is called  the {\em factorization} of $B$.}

\noindent The  {\it switch-number}  of a formal  bracket  $B$ is the number $\mathfrak{s}_{_B}$ defined recursively  as:  
\[ 
\mathfrak{s}_{_B} := 1 \  \text{ if $\ell_{B}=1$}; \qquad  \qquad
\mathfrak{s}_{_{B}}:= 2\big(\mathfrak{s}_{_{B_1}}+\mathfrak{s}_{_{B_2}}\big)\  \text{ if $\ell_B\geq2$ and ${B}=[B_1,B_2]$.}
\] 
\noindent For instance, the switch-numbers of $[[X_3,X_4],[[X_5,X_6],X_7]]$ and $[[X_5,X_6],X_7]$ are $28$ and $10$, respectively. When  no  confusion may arise,  we  also speak of  `degree and  switch-number of  Lie brackets of vector fields'.

\noindent We will use the following notion of {\em admissible bracket pair}:

 \begin{definition}\label{classCB}    Let $c\geq 0$,  $\ell\ge 1$, $q\geq c+\ell$ be integers, let  $B = B(X_{c+1},\ldots,X_{c+\ell})$ be  an iterated  formal bracket  and let   ${\bf g}=(g_1,\ldots,g_q )$  be a string of continuous vector fields. We say that  ${\bf g}$  is {\it of class $C^B$}   if  there exist non-negative integers $k_1\dots,k_q$ such that, by the only information that $g_i$ is of class  $C^{k_i}$ for every $i=1,\dots,q$  , one can deduce that 
$B({\bf g})$ is a  $C^0$ vector field  (see \cite[Def. 2.6]{FR2}). In this case,  we call $(B,{\bf g})$ an {\em admissible bracket pair} (of degree $\ell$ and switch number $\mathfrak{s}:=\mathfrak{s}_{_B}$).
  \end{definition}
 
 \noindent  For instance, if   $B= [ [[X_3,X_4],[X_5,X_6]],X_7 ]$ and  ${\bf g}=  
  (g_1,g_2,g_3,g_4,g_5,g_6,g_7,g_8)$,  then ${\bf g}$  is of class $C^B$ provided   $g_3,g_4,g_5,g_6$ are of class  $C^{3}$ and $g_7\in C^{1}$.

\section{Degree-$k$  sample stabilizability  with regulated cost} \label{section_mrf}
 Let us introduce   the  definitions of    degree-$k$ feedback generator,  sampling process,   and degree-$k$ sample stabilizability with regulated cost.
\subsection{Admissible trajectories}
In the following, we consider a control  set $A\subset\R^m$, a target $\T\subset\R^n$, a Lagrangian $l:\R^n\times A\to \R_{\ge0}$, and   vector fields $f_1,\dots, f_m:\R^n\to\R^n$. Furthermore, we define the function $\d:\R^n\to\R_{\geq0}$ as
\[
\d(x) = \inf_{y\in\T} |x-y| \qquad \forall x\in\R^n.
\]

\begin{definition}[Admissible  controls, trajectories and costs]\label{Admgen}   We say that  $(\alpha,y)$  is an {\em admissible control-trajectory pair}  if there exists $S_y\le +\infty$ such that:
	  \begin{itemize}
			\item[(i)] the control  $\alpha:[0,S_y[\to A$  is Lebesgue measurable; 
			\vsm
			\item[(ii)]  $y:[0,S_y[\to  \R^n\setminus \T$  is a  (Carathéodory)   solution   of the   control system
			\begin{equation}\label{control_sys}
				\dot y(s)= \sum_{i=1}^m  f_i(y(s))\, \alpha^i(s), 
			\end{equation}
		
satisfying,  if $S_y<+\infty$, $\lim_{s\to S_y^-}\d(y(s))=0$.  
	\end{itemize}
	Given an admissible pair $(\alpha,y)$,  we say that $(\alpha,y,\mathfrak{I})$  is an {\em admissible control-trajectory-cost triple} if
	$\mathfrak{I}:[0,S_y[\to  \R$ is the {\it integral cost},  given by 
			\begin{equation}\label{Pgen}
				\mathfrak{I}(s):= \int_ 0^{s }  l(y(\sigma),\alpha(\sigma))\, d\sigma, \quad \forall   s\in[0,S_y[.
			\end{equation}
	For every $x\in\R^n \setminus \T$,  we call   $(\alpha,y)$ and $(\alpha,y,\mathfrak{I})$ as above with $y(0)=x$,  an {\em admissible pair} and an {\em admissible triple  from $x$}, respectively.
	\vsm		
	For any admissible pair or triple   such that  $S_y<+\infty$, we  extend   $\alpha$, $y$, and	$\mathfrak{I}$  to $\R_{\geq0}$   by setting $\alpha(s):=\bar a$,  $\bar a\in A$ arbitrary,  and   $ (y,	\mathfrak{I})(s):= \lim_{\sigma\to S_y^-} (y(\sigma),	\mathfrak{I}(\sigma)),$  for any $s\geq S_y$.\footnote{In view of hypotheses {\bf (H1)}, {\bf (H2)} below, this limit always exists.}
\end{definition}

Throughout the whole paper, $k\ge1$ will be a given integer and
we will  consider the following sets of hypotheses.
{\em  
\begin{itemize}
\item[{\bf (H1)}]
The set  $A=\{\pm e_1, \dots, \pm e_m\}$\footnote{The vectors $e_1,\dots,e_m$ denote the elements of the canonical basis of $\R^m$.}
 and  $\T$ is closed  with compact boundary.
 \vsm
\item[{\bf (H2)}]  For any $a\in A$, $l(\cdot,a)$   is  locally  Lipschitz continuous on $\R^n$. Furthrmore,   $f_1,\dots, f_m$ belong to   $C^{k-1,1}_{b}(\Omega)$ for any bounded, nonempty subset $\Omega\subset\R^n$.  
 \end{itemize}}

 
 \subsection{Degree-k feedback generator} \label{def_control}
Let us introduce the sets of admissible bracket pairs associated with the vector fields $f_1,\dots,  f_m$  in the dynamics.  
 
\begin{definition}[Control label] \label{Lie_algebra} For any integer $h$ such that $1\le h\le k$,   let us define the set 
$\F^{(h)}$ of {\it control labels of degree $\leq h$} as 
\[ 
\F^{(h)} := \left\{ 
(B, \mathbf{g},\sgn) \ \left| 
\begin{array}{l} 
\sgn\in\{+,-\}  \text{ and $(B, \mathbf{g})$  is an admissible bracket pair} \\
 \text{ of degree $l_B\le h$ such that $\mathbf{g}:=(g_1,\ldots,g_q)$ satisfies} \\
 \text{ $g_j\in \{f_1,\dots,f_m\}$ for any $j=1,\dots,q$}
\end{array}
\right. 
\right\}.
\]
We will call {\em degree} and {\em switch number  of a control label} $(B, \mathbf{g},\sgn)\in \F^{(h)}$,  the degree and the switch number of $B$, respectively.  
 \end{definition}
 With any control label  in $\F^{(k)}$ let us associate an {\em oriented control,} defined as follows:
		\begin{definition}[Oriented control]\label{orcon} Consider a time $t>0$ and two  triples $(B, \mathbf{g},+), (B, \mathbf{g},-)  \in \F^{(k)}$,  
		  and let us define  the corresponding {\em oriented controls} ${\alpha}_{(B,{\bf g},+),t}$,  ${\alpha}_{(B,{\bf g},-),t}$, respectively, by means of the following recursive procedure:
	 	\begin{itemize} 
			\item[(i)] if $\ell_B=1$, i.e. $B=X_j$   for some integer $j\geq 1$,  we set
			$${\alpha}_{(B,{\bf g},+),t}(s):= e_i \qquad  \text{for any $s\in [0,t]$,}
		$$
		where   $i\in\{1,\dots,m\}$  is such that $B(\mathbf{g}) = f_i$   (i.e. $g_j=f_i$);
		\vsm
		\item[(ii)] if $\ell_B\geq 1 $,  we set $ {\alpha}_{(B,{\bf g},-),t}(s):=  -{\alpha}_{(B,{\bf g},+),t}(t-s)$ for any $s\in [0,t]$;
		\vsm
			\item[(iii)] if $\ell_B\ge2$ and $B=[B_1,B_2]$, we set $\mathfrak{s}_1:=\mathfrak{s}_{B_1}$,  $\mathfrak{s}_2:=\mathfrak{s}_{B_2}$, and  \newline $\mathfrak{s}:=\mathfrak{s}_B(=2\mathfrak{s}_1+2\mathfrak{s}_2)$ and,  for any $\sigma\in [0,t]$, we posit	
				\end{itemize} 
\bel{definizione_control}
{\alpha}_{(B,{\bf g},+),t}(\sigma):=
\begin{cases}
{\alpha}_{(B_1,{\bf g},+),\frac{\mathfrak{s}_1}{\mathfrak{s}}t}(\sigma)   \quad  &\text{if }\sigma\in \left[0,  \frac{\mathfrak{s}_1}{\mathfrak{s}}t \right[ \vspace{0.1cm}\\ 
{\alpha}_{(B_2,{\bf g},+),\frac{\mathfrak{s}_2}{\mathfrak{s}}t} \left(\sigma - \frac{\mathfrak{s}_1}{\mathfrak{s}}t \right)   \quad &\text{if }\sigma\in \left[\frac{\mathfrak{s}_1}{\mathfrak{s}}t,  \frac{\mathfrak{s}_1 + \mathfrak{s}_2}{\mathfrak{s}}t \right[ \vspace{0.1cm}\\
{\alpha}_{(B_1,{\bf g},-),\frac{\mathfrak{s}_1}{\mathfrak{s}}t}   \left(\sigma -  \frac{\mathfrak{s}_1 + \mathfrak{s}_2}{\mathfrak{s}}t \right) \,\,\, &\text{if }\sigma\in \left[ \frac{\mathfrak{s}_1 + \mathfrak{s}_2}{\mathfrak{s}}t,  \frac{2\mathfrak{s}_1 + \mathfrak{s}_2}{\mathfrak{s}}t \right[\vspace{0.1cm} \\
{\alpha}_{ (B_2,{\bf g},-),\frac{\mathfrak{s}_2}{\mathfrak{s}}t}  \left(\sigma - \frac{ 2\mathfrak{s}_1 + \mathfrak{s}_2}{\mathfrak{s}}t \right)  \,  &\text{if }\sigma\in \left[ \frac{ 2\mathfrak{s}_1 + \mathfrak{s}_2}{\mathfrak{s}}t, t \right].
\end{cases}
\eeq 
	
	\end{definition}

	\begin{example}
If $t>0$, $B=[X_3, [X_4,X_5]]$, and $\mathbf{g}=(f_3,f_2,f_1,f_2,f_3)$,  one has 	  
 \[
{\alpha}_{(B,{\bf g},+),t}(\sigma)\,\,=
\begin{cases}
\begin{aligned}
& e_1 \qquad &\text{if } \sigma\in [ 0, t/10 [   \\
& e_2 \qquad &\text{if } \sigma\in[ t/10 , 2t/10  [  \cup [ 7t/10, 8t/10  [  \\
& e_3 \qquad & \text{if } \sigma\in[2t/10 ,3t/10  [  \cup [ 6t/10, 7t/10  [  \\
& -e_1 \qquad &\text{if } \sigma\in[5t/10,6t/10[    \\
& -e_2 \qquad &\text{if } \sigma\in[ 3t/10 ,4t/10[  \cup [9t/10, t ]  \\
& -e_3 \qquad & \text{if } \sigma\in[4t/10 ,5 t/10[  \cup [8t/10,9t/10[ 
\end{aligned}
\end{cases} 
\]
and
\[ {\alpha}_{(B,{\bf g},-),t}(\sigma)=
\begin{cases}
\begin{aligned}
	& e_2 \qquad &\text{if } \sigma\in [ 0, t/10[ \cup[6 t/10,7 t/10[  \\
	& e_3 \qquad &\text{if } \sigma\in[ t/10 ,2 t/10[  \cup  [5 t/10, 6t/10[  \\
	& e_1 \qquad & \text{if } \sigma\in [4t/10 ,5t/10[   \\
	& -e_3 \qquad &\text{if } \sigma\in[ 3t/10, 4 t/10[ \cup  [7 t/10,8t/10[   \\
	& -e_2 \qquad &\text{if } \sigma\in [2t/10, 3t/10[\cup [8t/10,9t/10[   \\
	& -e_1 \qquad & \text{if } \sigma\in [9t/10 , t  ]  .
\end{aligned}
\end{cases} 
 \]
\end{example}

%
%

	\begin{definition}[Degree-$k$ feedback generator] \label{feedback}
		We call {\em degree-$k$ feedback generator} any map $\V:\R^n\setminus\T\to \F^{(k)}$  and write 
		 \bel{les}
\begin{array}{l}
\quad\qquad x\mapsto  \mathcal{V}(x):=\left(B_x,\mathbf{g}_x,\sgn_{\text{\tiny $x$}}\right),   \\[1.5ex]
 \ell(x) := \ell_{_{B_x}}, \qquad \mathfrak{s}(x) :=\mathfrak{s}_{_{B_x}}\quad \forall x\in \R^n\setminus\T.  
 \end{array}
 \eeq
\end{definition}
\begin{definition}[Multiflow] \label{multiflow} 
	Let a degree-$k$ feedback generator  $\mathcal{V}$ be given.    For every  $x\in\R^n\setminus\T$  and $t>0$,  let us define  the control
	$	\alpha_{x,t}:\R_{\geq0}\to A$ by setting
	$$
	\alpha_{x,t}(s):={\alpha}_{\mathcal{V}(x),t}(s\wedge t) \qquad \text{for any  $s\in\R_{\geq0}$}.
	$$
The maximal  solution\footnote{By maximal solution we mean  the solution to \eqref{ytx} defined on the largest subinterval of $\R_{\ge0}$ that contains $0$.}    to the Cauchy problem 
	\bel{ytx}
	\dot y(s)= \sum_{i=1}^m f_i\big(y(s)\big)\alpha^i_{x,t}(s), \qquad
		y(0)=x,   
		\eeq 
		will be called the  {\em $\mathcal{V}$-multiflow starting from $x$ up to the time $t$}  (or simply   {\em $\mathcal{V}$-multiflow}, when $x$ and $t$ are clear from the context) and will be denoted by $y_{x,t}$.
\end{definition}

\begin{remark} \label{rem_asym}
The relevance of the above construction is due to the asymptotic formula \cite[Thm. 3.7]{FR2}. In particular, under hypothesis {\bf (H2)} it is possible to prove that,  given a degree-$k$ feedback generator $\V$, for any $R>0$ there exist $\delta>0$ and $\w>0$ such that, for every $x\in\R^n$ with $\d(x)\leq R$ and every $t\in[0,\delta]$,    each $\V$-multiflow $y_{x,t}$ starting from $x$ up to the time $t$ satisfies
\[
\Big| y_{x,t}(t)-x- \sgn_x B_x({\bf g}_x)(x) \Big| \leq \w t \bigg(\frac{t}{{\mathfrak s}(x)}\bigg)^{\ell(x)}.
\]
The previous inequality shows that a $\V$-multiflow $y_{x,t}$ allows us to move approximately in the direction of the vector   field $\sgn_x B_x({\bf g}_x)$ evaluated at $x$, with an error which is proportional to $t^{\ell(x)+1}$.
\end{remark}

\subsection{
Sampling processes}
 We call $\pi:=\{ s_j \}_{j}$    {\em a partition of $\R_{\ge0}$} if $s_0=0$, $s_j  < s_{j+1}$ for any $j\in\N$,  and $\ds\lim_{j\to+\infty}s_j=+\infty$. The {\em sampling time}, or   {\em diameter,} of $\pi$  is the supremum of $\{s_{j+1}-s_j: \ \ j\in\N\}$.
 
\begin{definition}[$\V$-sampling  process-cost]\label{k_traj}
Given a degree-$k$ feedback generator $\V:\R^n\setminus\T \to \F^{(k)}$,   we refer to 
$(x, \pi, \alpha_x^\pi, y_x^\pi)$ as a {\em $\V$-sampling process}  if $x\in\R^n\setminus\T$, $\pi:=\{s_j\}_j$ is a partition of $\R_{\geq0}$,   $y_x^{\pi}$ is a continuous function  taking values in $\R^n\setminus\T$  defined  recursively by  
\bel{traj_generalizzate}
\begin{cases}
y_x^{\pi}(s) := y_{x_j,t_j}(s-s_{j-1}) \qquad \text{ for all   $s\in [s_{j-1}, \sigma_j[$ \, and \, $1\leq j \leq\J$} \\
y_x^{\pi}(0) = x,
\end{cases}
\eeq
where,  for all $j\geq 1$, $y_{x_j,t_j}$ is a $\V$-multiflow with $t_j:= s_j-s_{j-1}$, 
  $x_j:= y_x^{\pi}(s_{j-1})$ for all $1\leq j \leq \J$,  and\footnote{We mean $\J=+\infty$    if the set is empty.}
$$
\begin{array}{l}
 \sigma_j :=  \sup\left\{ \sigma\ge s_{j-1}  \text{ : } y_{x_j, t_j} \text{ is defined on } [s_{j-1},\sigma[, \  y_{x_j, t_j}([s_{j-1},\sigma[)\subset\R^n\setminus\T \right\},\\[1.0ex]
\J :=\inf\{j: \ \sigma_j\le s_j\}.
\end{array}
$$
We will refer to the map $y_x^\pi$, which is defined on the maximal interval $[0, \sigma_{\J }[$,  as  a {\em $\V$-sampling trajectory}.
 According to Def. \ref{multiflow}, the  corresponding  {\em $\V$-sampling control} $\alpha_{x}^{\pi}$ is defined as
\[
\alpha_x^{\pi}(s) := \alpha_{x_j, t_j}(s - s_{j-1}) \qquad \text{for all }s\in[s_{j-1}, s_j[\cap[0, \sigma_{\J }[, \quad 1\leq j\le \J.
\]
Furthermore, we define the {\em $\V$-sampling cost} $\mathfrak{I}_x^{\pi}$ as
\begin{equation}\label{Scostgen}
		\mathfrak{I}_x^{\pi}(s):=\int_0^s   l(y_x^{\pi}(\sigma),\alpha_x^{\pi}(\sigma))\, d\sigma,  \quad \forall s\in[0, \sigma_{\J }[,
			\end{equation}
and we call $(x, \pi, \alpha_x^\pi, y_x^\pi, \mathfrak{I}_x^\pi)$  a {\em $\V$-sampling process-cost}.
\end{definition}
 If  $(\alpha_x^{\pi},y_x^{\pi})$ [resp. $(\alpha_x^{\pi},y_x^{\pi}, \mathfrak{I}_x^{\pi})$] is an admissible pair [resp. triple] from $x$,   we say that the $\V$-sampling process $(x, \pi, \alpha_x^\pi, y_x^\pi)$ [resp. $\V$-sampling process-cost $(x, \pi, \alpha_x^\pi, y_x^\pi,\mathfrak{I}_x^{\pi})$] is {\em admissible}. In this case,
 when $ \sigma_{\J }=S_{y_x^\pi}<+\infty$, we extend $\alpha_x^{\pi}$, $y_x^{\pi}$, and $\mathfrak{I}_x^{\pi}$ to $\R_{\geq0}$,  as described  in Def. \ref{Admgen}.

\vsm

The   definition below  of {\it$\del$-scaled $\V$-sampling process-cost}, where $\V$ is a degree-$k$ feedback generator and  $\del:=(\delta_1,\ldots,\delta_k)\in\R^k_{>0}$, prescribes bounds on the steps  $s_j - s_{j-1}$ from above and from below, depending on the degree $\ell_j\in\{1,\dots,k\}$ of the involved formal bracket. 

 
\begin{definition}   [$\del$-scaled $\V$-sampling process-cost] \label{vconsistent}
		Let a vector  $\del:=(\delta_1,\ldots,\delta_k)$ in $\R^k_{>0}${, which we will  call {\it multirank},} be given,   and consider   a degree-$k$ feedback generator $\V:\R^n\setminus\T\to \F^{(k)}$. We refer to $(x, \pi, \alpha_x^\pi, y_x^\pi, \mathfrak{I}_x^\pi)$ [resp. $(x, \pi, \alpha_x^\pi, y_x^\pi)$] as a {\em $\del$-scaled $\V$-sampling process-cost}  [resp. {\em $\del$-scaled $\V$-sampling process}] provided  it is    a $\V$-sampling  process-cost [resp. $\V$-sampling process] such that the partition $\pi=\{s_j\}_j$ satisfies
		\bel{vconsistent_part}
		\Delta(k) \, \delta_{\ell_j} \leq s_j - s_{j-1}
		 \leq \delta_{\ell_j} \qquad \forall j\in\N, \ j\ge 1,
		\eeq
 where   $\ell_j := \ell(y_x^\pi (s_{j-1}))$ (see Def. \ref{les})  and  $\Delta(k):= \frac{k-1}{k}$.
\end{definition}

\begin{remark}\label{multirank}
When $k=1$, so that  $\Delta(1)=0$ and $\del=\delta\in\R_{>0}$,   the $\V$-sampling trajectory $y_x^\pi$ of a $\del$-scaled $\V$-sampling process-cost $(x, \pi, \alpha_x^\pi, y_x^\pi, \mathfrak{I}_x^\pi)$ is nothing but a standard $\pi$-sampling trajectory associated with a partition $\pi$ of sampling time smaller than $\delta$ (see for instance \cite{CLSS,CLRS,LM2}). 
 In fact, in this case  the degree-$1$ feedback generator $\V$ takes values in $\F^{(1)}$  and all the formal brackets $B$ of the elements $(B, \mathbf{g},\sgn)$ belonging to $\F^{(1)}$ have degree and switch number equal to 1. Accordingly,
 in view of Def. \ref{orcon},(i), each oriented control $\alpha_{x_j,t_j}$ included in the  $\V$-sampling control $\alpha_x^\pi$ (see Def. \ref{k_traj}), turns out to be constant on $[s_{j-1}, s_j]$.  
\end{remark}

\begin{remark}
When the system approaches the target by approximating the direction of a degree-$\ell$ Lie  bracket for a time $t$,   the distance from the target is reduced at most of a quantity which is proportional to $t^\ell$, as one can deduce from the estimate in Rem. \ref{rem_asym}.  Therefore, the lower bound in condition \eqref{vconsistent_part} 
 ensures that for each sampling trajectory the sum of the displacements is divergent, so that to obtain {\it uniform attractiveness} (see Def. \ref{samplestab_costo} and Rem. \ref{rem_attractiveness} below). 
We do observe that the lower bound vanishes when $k=1$, as in this case the sum of the displacements up to $i$ iterations is proportional to $s_i=\sum_{j=1}^i (s_j-s_{j-1})$, which diverges to $+\infty$ as $i\to+\infty$ by the very definition of partition of $\R_{\geq0}$.
\end{remark}


\subsection{Degree-$k$  sample stabilizability and global asymptotic controllability with regulated cost}
The notion  of  {\it degree-$k$  sample stabilizability of control system \eqref{control_sys} to $\T$ with  regulated cost} we will introduce in Def. \ref{samplestab_costo} below relies on the following notion 
 of {\it integral-cost-bound function}. 

%

\begin{definition}[Integral-cost-bound function]\label{ib}  
We call {\it integral-cost-bound function} any map ${\bf \Psi}:  \R_{>0}^3  \to\R_{\ge0}$, given by
$$
{\bf \Psi}(R,v_1,v_2):=\Lambda (R)\cdot\Psi(v_1,v_2),
$$
where 
\begin{itemize}
\item[(i)] the function $\Lambda:\R_{>0}\to\R_{>0}$   is  continuous, increasing, and  $\Lambda\equiv 1$ if $k=1$;
\vsm
\item[(ii)]  the function  $\Psi:  \R_{>0}^2  \to\R_{\ge0}$ is  continuous, 
 increasing and unbounded in the first variable,   strictly decreasing in the second variable;
\vsm
\item[(iii)]   there exists a  strictly decreasing bilateral sequence $(u_i)_{i\in\Z}\subset\R_{>0}$, such that, for some (hence, for any)  $j\in\Z$, one has
	\bel{ppsi}  \sum_{i=j}^{+\infty}\Psi(u_{i},u_{i+1})<+\infty,\quad\text{and} \lim_{i\to-\infty} u_i  = +\infty,  \quad  \lim_{i\to+\infty} u_i  = 0.  
	\eeq
	\end{itemize}	
\end{definition}

\begin{definition}[Degree-$k$ sample stabilizability with regulated cost]
\label{samplestab_costo}
 Let $\V:\R^n\setminus\T\to\F^{(k)}$ be  a  degree-$k$ feedback generator and let $U:\overline{\R^n\setminus\T}\to \R_{\geq0}$ be a  continuous, proper,   and positive definite function.  We  say that $\V$ {\it degree-$k$  $U$-sample stabilizes control system \eqref{control_sys} to  $\T$} if 
there exist a multirank map $\del:\R_{>0}^2\to \R_{>0}^k$  such that, for any $0<r<R$,   every $\del(R,r)$-scaled $\V$-sampling process $(x, \pi, \alpha_x^\pi, y_x^\pi)$ with $\d(x)\leq R$ is admissible and satisfies
		$$\begin{array}{l}
		 \text{(i)} \ {\bf d}(y_x^{\pi}(s)) \leq \Gamma(R)\quad \forall s \geq 0,
		\\[1.5ex]
		 \text{(ii)} \
			 \mathbf{t}(y_x^\pi,r):=\inf\big\{s\geq0 \text{ : } U(y_x^\pi(s))\leq {\varphi}(r) \big\}\leq {\bf T}(R,r),
											\\[1.5ex] 
	 \text{(iii)} \ \text{if $\exists\tau>0$:} \ U(y_x^\pi(\tau))\leq {\varphi(r)}, \  \text{then}   \  \d(y_x^\pi(s))\leq r \  \ \forall s\geq\tau,
		\end{array}$$
		where $\Gamma:\R_{\ge0}\to\R_{\ge0}$,   $\varphi:\R_{\geq0}\to\R_{\geq0}$ are continuous, strictly increasing and unbounded    functions  with $\Gamma(0)=0$, $\varphi(0)=0$,  and $\mathbf{T}:\R_{>0}^2\to \R_{>0}$ is a function increasing in the first variable and decreasing in the second one.
	We will refer to property (i) as {\em Overshoot boundedness}  and to (ii)-(iii) as {\em $U$-Uniform attractiveness}. 
	
		If, in addition, there exists an 
  integral-cost-bound function
${\bf \Psi}: \R_{>0}^3\to \R_{ \ge0} $   such  
	 that the $\V$-sampling process-cost $(x, \pi, \alpha_x^\pi, y_x^\pi,\mathfrak{I}_x^\pi)$ associated with the $\V$-sampling process $(x, \pi, \alpha_x^\pi, y_x^\pi)$ above  satisfies the inequality
	 $$
	 \,\,\,\begin{array}{l}
\text{ (iv)}\,\,\,\, \ds \mathfrak{I}_x^\pi(\mathbf{t}(y_x^\pi,r)) = 	\int_0^{\mathbf{t}(y_x^\pi,r)}   l(y_x^{\pi}(s),\alpha_x^{\pi}(s))\, ds   \leq  {\bf \Psi}\Big(R,\, U(x), \,U(y_x^\pi(\mathbf{t}(y_x^\pi,r)))\Big),  
\end{array}
	 $$ 
we say that  $\V$ {\em  degree-$k$  $U$-sample stabilizes control system \eqref{control_sys} to $\T$ with ${\bf \Psi}$-regulated cost}. We will refer to (iv) as {\em Uniform cost  boundedness}.

  When there exist some function $U$,  some {degree}-$k$ feedback generator $\V$ [and an integral-cost-bound function ${\bf \Psi}$] such that  $\V$ degree-$k$ $U$-sample stabilizes control system  \eqref{control_sys} to  $\T$ [with ${\bf \Psi}$-regulated cost], we  say that  system  \eqref{control_sys} is {\em degree-$k$ $U$-sample stabilizable to $\T$}  [{\em with ${\bf \Psi}$-regulated cost}]. Sometimes, we will simply say that system  \eqref{control_sys} is {\em degree-$k$  sample stabilizable to $\T$}  [{\em with regulated cost}]. 
\end{definition}

\begin{remark}\label{rem_attractiveness} Using the notations of Def. \ref{samplestab_costo}, the  $U$-uniform attractiveness immediately implies  the following, standard uniform attractiveness condition   
\[
\exists\, {\bf S}(R,r)>0 \quad \text{such that} \quad  \d(y_x^\pi(s)) \leq r \qquad \text{for all $s\geq {\bf S}(R,r)$}
\]
(with ${\bf S}(R,r)\le {\bf T}(R,r)$), that characterizes the classical  notion of sample stabilizability (see e.g. \cite{CLSS}).  Similarly,  for $k=1$ (so that $\Lambda\equiv 1$)  the {\it uniform cost boundedness} condition above implies the cost  bound condition considered in the notion of sample stabilizability with regulated cost first introduced in \cite{LM}. Indeed, from Def. \ref{samplestab_costo}, (iv)   it follows that  
\[
\int_0^{{\bf s}(y_x^\pi,r)} l(y_x^{\pi}(s),\alpha_x^{\pi}(s))\, ds   \leq \Psi(U(x),\varphi(r))\le \Psi(U(x),0)=:W(x),
\]
where ${\bf s}(y_x^\pi,r):= \inf\{s\geq0 \text{ : } \d(y_x^\pi(\sigma))\leq r \ \text{ for all }\sigma\geq s\}$, since   ${\bf s}(y_x^\pi,r)\le   {\bf t}(y_x^\pi,r)$ by conditions (ii) and (iii) above.   Therefore, thanks to Rem. \ref{multirank},  it is easy to see that  degree-1 sample stabilizability with regulated cost implies the notion of sample stabilizability with regulated cost given  in \cite{LM}. In addition,  we will show in Sec. \ref{sec_comparison} below, that  the notion of degree-$k$ sample stabilizability (without cost regulation)  is in fact equivalent to previous definitions of sample stabilizability. 
\end{remark}

\begin{definition}[Global asymptotic controllability with regulated cost] \label{GAC_costo}
	Control system \eqref{control_sys} is said to be \textit{globally asymptotically controllable} (in short, GAC) {\em to $\T$} if,  for any   $0<r<R$,  there exist ${\bf \Gamma}= {\bf \Gamma}(R)>0$ and  ${\bf S}={\bf S}(R,r)>0$   with $\ds\lim_{R\to0} {\bf \Gamma}(R)=0$, such that   for every $x\in \R^n$ with $\d(x)\leq R$ there exists an admissible  control-trajectory pair $(\alpha,y)$ from $x$ that satisfies the following conditions (i)--(iii):
	$$
	\begin{array}{lllll}
		
		&{\rm (i)} &{\bf d}(y(s)) \leq {\bf \Gamma}(R)\qquad &\forall s \geq 0; \qquad\qquad\qquad &\text{({\it Overshoot boundedness})}\\[1.5ex]
		&{\rm (ii)} &{\bf d}(y(s)) \leq r \qquad &\forall s \geq {\bf S}(R,r);  &\text{({\it Uniform attractiveness})}\\[1.5ex]
		&{\rm (iii)} &\ds\lim_{s \to +\infty} {\bf d}(y(s)) = 0. &\, &\text{({\it Total attractiveness})}
	\end{array}
	$$
	If, moreover, there exists a  function ${\bf W}:\overline{\R^n\setminus\T}\to\R_{\geq0}$ continuous, proper, and positive definite, such that the admissible control-trajectory-cost triple $(\alpha,y,\mathfrak{I})$ associated with the admissible control-trajectory pair $(\alpha,y)$ above satisfies
	$$
	\begin{array}{l}
		\qquad{\rm (iv)} \,\,\, \ds \int_0^{S_y} l(\alpha(s),y(s)) \, ds \leq {\bf W}(x), \,\,\,\qquad\qquad \,\, \text{({\it Uniform cost boundedness})}
	\end{array}
	$$
	where $S_y\le +\infty$ is as in Def. \ref{Admgen}, we say that   control system  \eqref{control_sys} is  \textit{globally asymptotically controllable to $\T$ with} {\em${\bf W}$-regulated cost} (or simply, {\em with regulated cost}).
\end{definition}

 \section{Main result}\label{concludesec}
 In this section, we show that  degree-$k$ feedback stabilizability  with regulated cost implies global asymptotic controllability with regulated cost.   
\vsm

To this aim, given  an  integral-cost-bound  function ${\bf\Psi}=\Lambda\Psi$, let us introduce 
   the strictly increasing and unbounded map $ {\bf\Phi}:\R_{\geq0}\to\R_{\geq0}$  defined  by setting
\bel{Phi}
{\bf\Phi}(u):=
\begin{cases}
	\ds\Psi(u, u_{i+1}) + \sum_{j=i+1}^{+\infty} \Psi(u_{j},u_{j+1}) \qquad&\text{if $u\in ]u_{i+1},u_i]$  for some  $i\in\Z$,} \\
	0 \qquad \qquad&\text{if $u=0$,}
\end{cases}
\eeq
where $(u_i)_i$ is the sequence of Def. \ref{ib}. 
Since in Thm. \ref{sample-->gaccosto} below we will use ${\bf\Phi}$ as an upper bound for the cost, we can assume without loss of generality that ${\bf\Phi}$ is continuous and strictly increasing. By this we mean that we will   still denote by ${\bf\Phi}$ a continuous and strictly increasing approximation of it from above, such that $\ds\lim_{u\to0}{\bf\Phi}(u)=0$ (this limit follows from the properties of $\Psi$).


\begin{theorem}\label{sample-->gaccosto} Assume {\bf (H1)}, {\bf (H2)}. 
Then the following statements hold:
\begin{itemize}
\item[{\rm (i)}] If control system \eqref{control_sys} is degree-$k$ sample stabilizable to $\T$, then it is globally asymptotically controllable to $\T$.

 \item[{\rm (ii)}] If  control system \eqref{control_sys} is degree-$k$ $U$-sample stabilizable to $\T$ with ${\bf\Psi}$-regulated cost, then system \eqref{control_sys} is globally asymptotically controllable to $\T$ with ${\bf W}$-regulated cost,  the map  ${\bf W}$  being defined as  
\[
{\bf W}(x)= \Lambda(\varphi^{-1}(U(x)) {\bf \Phi}(U(x)),
\]
where $\Lambda$ is the same as in the definition of ${\bf\Psi}$,  $\varphi$ is as in Def. \ref{samplestab_costo}, and  ${\bf\Phi}$ is as in \eqref{Phi}.
\end{itemize}
\end{theorem}

\begin{proof}
  Let us begin with the proof of   (ii).  First of all, let $\V$, $U$, $\del$, $\Gamma$, $\varphi$, ${\bf T}$ and ${\bf\Psi}=\Lambda\Psi$ be as in Def. \ref{samplestab_costo}. 
Let  $(u_i)_{i\in\Z}\subset \R_{>0}$ be a strictly decreasing bilateral sequence with $\ds\lim_{i\to-\infty}u_i=+\infty$ and  $\ds\lim_{i\to+\infty}u_i=0$  associated with $\Psi$, as in Def. \ref{ib}, and define ${\bf\Phi}$ as in  \eqref{Phi}.

\noindent Notice that it is implicit in Def. \ref{samplestab_costo} that 
\bel{maggiore_dist}
\varphi^{-1}(U(x))\geq \d(x) \qquad \text{for any $x\in \R^n\setminus\T$.}
\eeq
Indeed, if this is not true  there exist  $x\in\R^n$ and $r>0$ such that $\varphi^{-1}(U(x))<  r< \d(x)$. Thus, for any $\del(\d(x),r)$-scaled $\V$-sampling process-cost $(x, \pi, \alpha_x^\pi, y_x^\pi, \mathfrak{I}_x^\pi)$, Def. \ref{samplestab_costo}, (iii) implies that $\d(y_x^\pi(s))\leq r$ for any $s\geq0$, even if $\d(x)=\d(y_x^\pi(0))>r$, a contradiction.

\vsm

\noindent
 {\it Step 1}.
Fix a pair $(R,r)$ with $0<r<R$, and $x\in\R^n\setminus\T$ with $\d(x)\le R$. Let $\bar \j(x)$ be the integer such that 
\[
U(x)\in ]u_{\bar \j(x)+1}, u_{\bar \j(x)}]  
\]
and define the bilateral sequence $(\tilde u_i)_{i\in\Z}$ as follows:
\[
\tilde u_0:=U(x), \qquad \tilde u_i := u_{\bar\j(x)+i} \ \ \forall i\geq1, \qquad \tilde u_{i} := u_{\bar\j(x)+i+1} \ \ \forall i\leq -1.
\]
Hence, consider the bilateral sequence $(\rho_i)_{i\in\Z}$ given by
\bel{rhoi}
\rho_i :=\varphi^{-1}(\tilde u_i) \qquad \text{for all $i\in\Z$.}
\eeq
The properties of $\varphi$ imply that $(\rho_i)_{i\in\Z}$ is strictly decreasing,\footnote{More precisely,    the only exception to  this strict monotonicity might consist in   $\rho_0=\rho_{-1}$, as soon as $U(x)= \tilde u_0= u_{\bar \j(x)}$.} with $\ds\lim_{i\to-\infty}\rho_i=+\infty$ and  $\ds\lim_{i\to+\infty}\rho_i=0$, and from \eqref{maggiore_dist}, \eqref{Lzeta},  and the previous estimates it follows that  
\bel{r0}
\d(x)\leq \rho_{0} =\varphi^{-1}(U(x)) \leq \varphi^{-1}( d_{U_-}^{-1}(R))
 =:\zeta(R),
\eeq
where $d_{U_-}$ is as in \eqref{zeta}. Moreover, the properties of $\Gamma$ imply that
\bel{Gamma_decrescente}
\lim_{i\to-\infty}\Gamma(\rho_i)=+\infty, \qquad \lim_{i\to+\infty}\Gamma(\rho_i)=0.
\eeq
Now, define $\i(r)$ and $\i(R)$ as the  integers such that 
\bel{ir}
\varphi(\Gamma^{-1}(r))\in]\tilde u_{\i(r)},\tilde u_{\i(r)-1}], \qquad d_{U_-}^{-1}(R) \in]\tilde u_{\i(R)+1},\tilde u_{\i(R)}],
\eeq
so that 
$
\i(r)-\i(R)>0
$
is the number of strips $\{z\in\R^n: \tilde u_{i+1}< U(z)\le \tilde u_i\}$ such that $i\in\{\i(R),\dots,\i(r)-1\}$. Clearly, $\i(R)\leq 0$, as $U(x)\leq d_{U_-}^{-1}(\d(x))\le d_{U_-}^{-1}(R)$. 

\vsm

\noindent
 {\it Step 2}.
Let us consider the control-trajectory pair $(\alpha,y)$ obtained by a recursive procedure as follows:
$$
\begin{cases}
x_1:=x, \qquad &\Ss_0:=0,\\
\Ss_i-\Ss_{i-1}= \Ss(y_{x_i}^{\pi_i},\rho_{i}), \qquad &x_{i+1}:=y(\Ss_{i}), \\
(\alpha,y)(s):= (\alpha^{\pi_i}_{x_i},y^{\pi_i}_{x_i})(s-\Ss_{i-1}) &\forall s\in[\Ss_{i-1},\Ss_i],
\end{cases}
$$
where, for any integer $i\ge 1$, $(x_i, \pi_i, \alpha_{x_i}^{\pi_i}, y_{x_i}^{\pi_i}, \mathfrak{I}_{x_i}^{\pi_i})$ is a $\del(\rho_{i-1},\rho_{i})$-scaled $\V$-sampling process-cost, and    $\Ss(y_{x_i}^{\pi_i},\rho_{i})$ is as in Def. \ref{samplestab_costo}, so that 
$$
U(x_{i+1})=U(y_{x_i}^{\pi_i}(\Ss(y_{x_i}^{\pi_i},\rho_{i}))=\tilde u_i=u_{\bar \j(x) +i}.
$$
 Indeed, in view of \eqref{r0}, it can be deduced by means of an induction argument that $\d(x_i)\le \rho_{i-1}$ for any $i\geq1$, so that $(x_i, \pi_i, \alpha_{x_i}^{\pi_i}, y_{x_i}^{\pi_i}, \mathfrak{I}_{x_i}^{\pi_i})$ satisfies conditions (i)--(iv) in Def. \ref{samplestab_costo} with $\rho_{i-1}$ and $\rho_{i}$ replacing $R$ and $r$, respectively.
As a consequence, by Def. \ref{samplestab_costo}, (i) and the monotonicity of $\Gamma$ we obtain
\bel{2est}
\d(y(s))\le \Gamma(\rho_{i-1}) \qquad \text{for any $s\ge \Ss_{i-1}$.}
\eeq
for any $i\geq1$.
Now we set $S_y:=\ds\lim_{i\to+\infty} \, \Ss_i$ and we associate with $(\alpha,y)$ its corresponding cost ${\mathfrak I}$, as in \eqref{Pgen}.
By \eqref{Gamma_decrescente} together with \eqref{2est} we deduce that 
\[
\lim_{s\to S_y^-} \d(y(s))=0,
\]
which in turn implies that, on the one hand,  $(\alpha,y,{\mathfrak I})$ is an admissible triple, and, on the other hand,  the total attractiveness property (iii) in Def. \ref{GAC_costo} holds true. Furthermore, \eqref{2est} and \eqref{r0} yield the overshoot boundedness property (i) in Def. \ref{GAC_costo}.  Actually, in view of the properties of $\Gamma$, $\varphi$, and $d_{U_-}$, one has 
\[
\d(y(s))\leq \Gamma(\rho_{0})\leq  
 \Gamma(\zeta(R))
=:{\bf \Gamma}(R) \qquad \forall s\geq0, \qquad
\ds\lim_{R\to0} {\bf \Gamma}(R)=0.
\]
By the very definition of $(\rho_i)_i$ and by \eqref{ir} one gets $\rho_{\i(r)}\leq \Gamma^{-1}(r)$, so that, again by \eqref{2est}, it holds
\[
\d(y(s))\le \Gamma(\rho_{\i(r)})\le r \qquad \text{for any $s\ge \Ss_{\i(r) \vee0}$.} 
\]
This yields the uniform attractiveness property (ii) in Def. \ref{GAC_costo}. 
 Indeed, recalling that $\Ss(y_{x_i}^{\pi_i},\rho_{i})\leq {\bf T}(\rho_{i-1},\rho_{i})$, $\i(R)\le 0$,  $\i(R)<\i(r)$ and the monotonicity properties of ${\bf T}$ one obtains
\[
\Ss_{\i(r)\vee0} =\sum_{i=1}^{\i(r) \vee0}[\Ss_i-\Ss_{i-1}] 
\leq \sum_{i=\i(R)+1}^{\i(r) \vee0} {\bf T}(\rho_{i-1},\rho_{i}) \leq \big(\i(r)-\i(R)\big) {\bf T}(\rho_{\i(R)}, \rho_{\i(r)})=: {\bf S}(R,r).
\]
We do observe that the indexes $\i(R)$ and $\i(r)$ depend on $x$. Conversely, the difference $\i(r)-\i(R)$ as well as $\rho_{\i(R)}$ and $\rho_{\i(r)}$ are independent of $x$, but depend only on $R$, $r$ and the given bilateral sequence $(u_i)_{i\in\Z}$. Therefore, ${\bf S}$ depends only on $R$ and $r$.

\vsm

\noindent
{\it Step 3}.
In order to conclude, it remains  to prove the uniform cost boundedness property (iv) in Def. \ref{GAC_costo}.
By construction, we have $U(x_i) =\tilde u_i= u_{\bar j(x)+i-1}$ for any $i\geq2$ 
so that, by \eqref{ppsi}, \eqref{Phi}, and the properties of the function $\Lambda$ we get
\[
\begin{split}
\int_0^{S_y}&l(y(s),\alpha(s))\,ds  = \sum_{i=1}^{+\infty} \int_{\Ss_{i-1}}^{\Ss_i} l(y^{\pi_i}_{x_i}(s-\Ss_{i-1}),\alpha^{\pi_i}_{x_i}(s-\Ss_{i-1}))\,ds \\
&=\int_0^{\Ss(y_{x_1}^{\pi_1},\rho_{1})} l(y^{\pi_1}_{x_1}(s),\alpha^{\pi_1}_{x_1}(s)) ds +\sum_{i=2}^{+\infty} \int_0^{\Ss(y_{x_i}^{\pi_i},\rho_{i})} l(y^{\pi_i}_{x_i}(s), \alpha^{\pi_i}_{x_i}(s)) ds \\
&\le  \Lambda(\rho_{0}) \Psi(U(x),u_{\bar\j(x)+1})+
 \sum_{i=2}^{+\infty} \Lambda(\rho_{i-1})\Psi(u_{\bar\j(x)+i-1},u_{\bar\j(x)+i}) \\
&\le \Lambda(\rho_0){\bf\Phi}(U(x))=\Lambda(\varphi^{-1}(U(x)){\bf\Phi}(U(x))={\bf W}(x).
\end{split}
\]
We point out that the fact that ${\bf W}$ is continuous, proper and positive definite follows from the properties of $\Lambda$, $\varphi$, $U$ and $\Psi$ (hence, ${\bf\Phi}$).
Statement (ii) of Thm. \ref{sample-->gaccosto} is thus proved.  
\vsm
Now we prove statement (i). 
Let us take an arbitrary strictly decreasing bilateral  sequence $(\tilde\rho_i)_{i\in\Z}$ with $\ds\lim_{i\to-\infty}\tilde\rho_i=+\infty$ and  $\ds\lim_{i\to+\infty}\tilde\rho_i=0$.
Given $0<r<R$ and $x\in\R^n\setminus\T$ with $\d(x)\le R$, let $\bar\j(x)$ be the integer such that $\d(x)\in]\tilde\rho_{\bar\j(x)+1},\tilde\rho_{\bar\j(x)}]$. Set 
$$
\rho_0:=\d(x), \qquad \rho_i:=\tilde\rho_{\bar\j(x)+i} \ \ \text{for any  $i\ge 1$,} \qquad  \rho_{i}:= \tilde\rho_{\bar\j(x)+i+1} \ \ \text{for any $i\le -1$.}
$$
Hence, define $\i(R)$ as the integer such that $R\in]\rho_{\i(R)+1}, \rho_{\i(R)}]$ and $\i(r)$ as the integer such that 
$\Gamma^{-1}(r)\in]\rho_{\i(r)},\rho_{\i(r)-1}]$ and  set $\zeta(R):=\rho_{\i(R)}$ (so that, clearly,  $\rho_0\leq\zeta(R)$). From now on  the proof proceeds exactly as in the {\em Step 2} above. 
\end{proof}

\begin{remark} \label{rem_k=1}
When the degree $k$  is equal to 1, a degree-$1$ feedback generator $x\mapsto \V(x)\in\F^{(1)}$ coincides with an usual feedback function $x\mapsto \alpha(x)\in A$ and any oriented control is constant in the whole time interval of definition. Furthermore, any $\del$-scaled $\V$-sampling process $(x, \pi, \alpha_x^\pi, y_x^\pi)$ is nothing but a standard $\pi$-sampling process associated with a partition $\pi$ of sampling time smaller than $\delta$ (see  Rem. \ref{multirank}). In this case, no control-linear structure of the dynamics function is needed in order to establish the above result. In particular, for $k=1$   the proof of Thm. \ref{sample-->gaccosto} can be easily adapted to general control systems of the form 
\bel{gen_control_sys}
\dot y(s)= F(y(s),\alpha(s)) , \qquad \alpha(s)\in A ,
\eeq
where $A\subset \R^m$ is a nonempty compact subset and $F:\R^n\times A \to \R^n$ is continuous in both variables and locally Lipschitz continuous in the state variable, uniformly with respect to the control variable.
\end{remark}



 \section{Comparing different notions of sample stabilizability}\label{sec_comparison}
 
The known notions of sample stabilizability, both in the first order case (see \cite[Def. I.3]{CLSS}) and in the higher order case (see \cite[Def. 2.18]{Fu}),   look rather simpler than the notion of degree-$k$ $U$-sample stabilizability proposed in Def. \ref{samplestab_costo}. Nevertheless, we show in this section that, in the absence of a cost, these definitions are in fact equivalent. 

\vsm
 For the reader's convenience let us recall the definition proposed in \cite{Fu} that, for $k=1$, coincides with the classical definition in \cite{CLSS}, as observed in \cite[Rem. 2.20]{Fu}. To this aim, given $\delta>0$, we say that a partition $\pi=\{s_j\}_j$ of $\R_{\ge0}$  is a {\em degree-$k$ partition of rank $\delta$} if, for any integer $j\geq1$, it holds 
\bel{RRR}
\frac{k-1}{2k}\cdot\frac{\delta}{\sqrt[k]{j}}\le s_j-s_{j-1}\le\delta. 
\eeq

\begin{definition}[{\cite[Def. 2.18]{Fu}}]\label{sample_stab}
Let $\V:\R^n\setminus\T\to\F^{(k)}$ be a degree-$k$ feedback generator. We say that $\V$  \textit{degree-$k$ sample stabilizes    control system \eqref{control_sys}  to $\T$}  if there exists a map $\delta:\R_{>0}^2\to \R_{>0}$  such that, for any $0<r<R$, for any degree-$k$ partition $\pi = \{ s_j \}_j$ of $\R_{\ge0}$ of rank $\delta(R,r)$ and for any $x \in \R^n\setminus\T$ with ${\bf d}(x) \leq R$, any $\V$-sampling process $(x,\pi,\alpha_x^{\pi},y_x^{\pi})$ is admissible and the following conditions hold:
	\begin{itemize}	
	\item[(i)] ${\bf d}(y_x^{\pi}(s)) \leq \Gamma(R)$ for all $s \geq 0,$$\qquad\qquad\qquad\,\,\,\,\,${\em (Overshoot boundedness)}
	\vsm
	\item[(ii)] ${\bf d}(y_x^{\pi}(s)) \leq r$ for all $s \geq {\bf S}(R,r),$$\quad\qquad\qquad\,\,\,$ {\em(Uniform attractiveness)}
		\end{itemize}
		where $\Gamma:\R_{\ge0}\to\R_{\ge0}$  is a continuous, strictly increasing and unbounded    function  with $\Gamma(0)=0$,  and $\mathbf{S}:\R_{>0}^2\to \R_{>0}$ is a function increasing in the first variable and decreasing in the second one.
\end{definition}
 
\begin{remark}
    The notion of stabilizability given in  Def. \ref{sample_stab} looks quite simpler than the one introduced in  Def. \ref{samplestab_costo}.  In the first instance, the uniform attractiveness in Def. \ref{samplestab_costo} posits  the existence of a  map $U$ which is not needed in Def. \ref{sample_stab}. Furthermore, the  partitions $\pi$ in Def. \ref{samplestab_costo} are associated with   multiranks $\del=(\delta_1,\ldots,\delta_k)$  verifying \eqref{vconsistent_part}, while in  Def. \ref{sample_stab} a single   rank (satisfying \eqref{RRR}) is involved.

Instead, in view of  Thm.  \ref{sample-->gaccosto} the two notions of  degree-$k$ stabilizability provided by   Def. \ref{samplestab_costo}  and Def. \ref{sample_stab}    turn out to be equivalent. Actually, they are both equivalent to the classical concept of sample stabilizability  (see \cite{CLSS}), as stated in Theorem \ref{L_equiv} below.\end{remark}

\begin{theorem}\label{L_equiv} Assume  {\bf (H1)}-{\bf (H2)}. Then, the following conditions are equivalent:
\begin{itemize}
\item[{\rm (i)}]   control system  \eqref{control_sys} is degree-$k$ $U$-sample stabilizable to $\T$ for some $k\ge 1$ and some function $U$, in the sense of 
Def. \ref{samplestab_costo};
\vsm
\item[{\rm (ii)}]  control system  \eqref{control_sys} is degree-$k$ sample stabilizable to $\T$ for some $k\ge 1$, in the sense of  Def. \ref{sample_stab};
\vsm
\item[{\rm (iii)}]   control system  \eqref{control_sys} is sample stabilizable to $\T$ (i.e.,  degree-$1$ sample stabilizable to $\T$, in the sense of  Def. \ref{sample_stab}).
\end{itemize}
\end{theorem}
\begin{proof}  (i)$ \ \Rightarrow \ $(iii). From Thm. \ref{sample-->gaccosto}, 
 degree-$k$  $U$-sample stabilizability of   control system  \eqref{control_sys} to $\T$ for some $k\ge1$ and some function $U$,  implies  GAC to $\T$. Under assumptions {\bf (H1)}-{\bf (H2)},  from 
 an  inverse Lyapunov theorem proved in \cite[Thm. 3.2]{KT04}, GAC to $\T$ is equivalent to the existence of a CLF $\tilde U$ which is locally Lipschitz continuous on $\R^n$. In view of \cite[Thm. 4.9]{LM}, this leads to the existence of a CLF   that is locally semiconcave on $\R^n\setminus\T$, so that  \cite[Thm. 3.1]{LM} finally implies  that  \eqref{control_sys} is sample stabilizable to $\T$ in the classical sense (see also  \cite{CLSS,CLRS,R1,R2}).  
 \vsmm
 \noindent  (iii)$ \ \Rightarrow \ $(i).   If  control system  \eqref{control_sys} is sample stabilizable to $\T$, then it is GAC to $\T$ by classical results, so that, as above, then there exists a CLF $U$ which is locally semiconcave in $\R^n\setminus\T$. Hence, from \cite[Thm. 1]{FMR2} it follows that the system is 
degree-$k$ $U$-sample stabilizable to $\T$ for $k=1$.
\vsmm
\noindent  (ii)$ \ \Leftrightarrow \ $(iii).
If  control system \eqref{control_sys} is degree-$k$ sample stabilizable to $\T$ for some $k\ge1$ in the sense of Def. \ref{sample_stab},  then it is also GAC to $\T$, by \cite[Prop. 4.2]{Fu}.  Once again, this implies that system  \eqref{control_sys} is sample stabilizable to $\T$. The converse implication is trivial since sample   stabilizability to $\T$ is in fact equivalent to 
degree-$k$ sample stabilizability to $\T$ for $k=1$, as defined in Def. \ref{sample_stab} (see \cite[Rem. 2.20]{Fu}).
\end{proof}

\end{document}